\documentclass[11pt]{amsart}
\usepackage{amsmath}
\usepackage{amssymb}
\usepackage{tabularx}
\usepackage{enumerate}
\usepackage{graphicx}
\usepackage{texdraw}

\usepackage{mathrsfs}

\usepackage{amsfonts,amssymb,amsmath}
\usepackage{epsfig}

\newtheorem{Theorem}{Theorem}[section]
\newtheorem{Lemma}[Theorem]{Lemma}

\newtheorem{Definition}{Definition}[section]

\numberwithin{equation}{section}
\numberwithin{figure}{section}

\begin{document}

\title[]{Existence of solutions to a generalized self-dual Chern-Simons equation on finite graphs}

\author[Y. Hu]{Yuanyang Hu$^1$}
\thanks{$^1$ School of Mathematics and Statistics,
	Henan University, Kaifeng, Henan 475004, P. R. China.}

\thanks{{\bf Emails:} {\sf yuanyhu@mail.ustc.edu.cn} (Y. Hu).}
\date{\today}


\begin{abstract}
	Let $G=(V,E)$ be a connected finite graph.
We study the existence of solutions for the following generalized Chern-Simons equation on $G$
\begin{equation*}
	\Delta u=\lambda \mathrm{e}^{u}\left(\mathrm{e}^{u}-1\right)^{5}+4 \pi \sum_{s=1}^{N} \delta_{p_{s}} \quad ,
\end{equation*}
where $\lambda>0$, $\delta_{p_{s}}$ is the Dirac mass at the vetex $p_s$, and $p_1, p_2,\dots, p_N$ are arbitrarily chosen distinct vertices on the graph. We show that there exists a critial value $\hat{\lambda}$ such that when $\lambda > \hat{\lambda}$, the
generalized Chern-Simons equation has at least two solutions, when $\lambda = \hat{\lambda}$, the
generalized Chern-Simons equation has a solution, and when $\lambda < \hat\lambda$, the
generalized Chern-Simons equation has no solution.
\end{abstract}


\maketitle

\section{Introduction}
 Abrikosov \cite{A} considered Magnetic vortex configurations in the context of Ginzburg–Landau theory of superconductivity. Later, Nielsen and Olesen stressed the relevance to high-energy physics of vortex-line solutions of the Abelian Higgs model in the context of dual string models \cite{NO}. Since then the interest on vortices has continued to grow both in condensed-matter and particle physics. A self-dual Chern-Simons model now plays an important role in various areas of physics, many researchers did a lot of significant work on the existence of vortices in this Chern-Simons model \cite{Bdd, ci, H, TY, Y}. Caffarelli and Yang \cite{CY} established the existence of condensates or periodic multivortices in the Abelian Chern-Simons-Higgs model. Bazeia, da Hora, dos Santos \cite{Bdd}, proposed a generalized Chern-Simons model, and obtained a generalized self-dual Chern-Simons equation. Later, Han \cite{H} established the existence of doubly periodic multi-vortices solutions to the generalized self-dual Chern-Simons model.
 
In recent years, increasing efforts have been devoted to the development of analysis on graph; see, for example, \cite{ ALY, Bdd, HLY, Hub, LP}  and the references therein. Bendito et al. \cite{BCE} considered the discrete analogues of boundary value problems for Laplacian operator on Riemannian manifolds. Lin and Wu \cite{LW} established the existence and nonexistence of global solutions for the following semilinear heat equation on G
\begin{equation*}
	\left\{\begin{array}{lc}
		u_{t}=\Delta u+u^{1+\alpha} & \text { in }(0,+\infty) \times V, \\
		u(0, x)=a(x) & \text { in } V,
	\end{array}\right.
\end{equation*}
where $G=(V,E)$ is a finite or locally finite connected weighted graph.  Bauer et al. \cite{BY} established the Li-Yau gradient estimate for the heat kernel on graphs. Horn et al. \cite{HLLY} proved Li-Yau-type estimates for bounded and positive solutions of the heat equation on graphs.  Grigor'yan, Lin and Yang \cite{ALY} studied Kazdan-Warner equation on a finite graph. Ge, Hua and Jiang \cite{GHJ} studied the Liouville equation $$-\Delta u=e^{u}$$ on a graph $G$ satisfying a certain isoperimetric inequality. In \cite{GJ}, Ge and Jiang studied the $1$-Yamabe equation
\begin{equation}
	\Delta_{1} u+g \operatorname{Sgn}(u)=h|u|^{\alpha-1} \operatorname{Sgn}(u),
\end{equation}
 on connected finite graph, where $\Delta_{1}$ is the discrete $1$-Laplacian, $\alpha>1$ and $g, h>0$ are known. In \cite{HLY},  Huang, Lin and Yau studied the mean field equation 
\begin{equation}\label{y}
	\Delta u+e^{u}=\rho \delta_{0}
\end{equation}  
 on $V$, where $G=(V,E)$ is a connected finite graph, constant $\rho >0$ and $\delta_{0}$ is the Dirac mass at the zero point, and proved an existence result to \eqref{y}, which is different from that on the cotinuous limit-the two dimensional flat tori, when $\rho= 8\pi$. 
 
 Let $G=(V,E)$ be a connected finite graph, where $V$ denotes the vetex set and $E$ denotes the edge set. In this paper, we estabilish the existence of solutions to following generalized self-dual Chern-Simons equation
\begin{equation}\label{E1}
	\Delta u=\lambda \mathrm{e}^{u}\left(\mathrm{e}^{u}-1\right)^{5}+4 \pi \sum_{s=1}^{N} \delta_{p_{s}} \quad \text { on } \quad V,
\end{equation}
where $\lambda>0$, $\delta_{p_{s}}$ is the Dirac mass at the vetex $p_s$, and $p_1, p_2,\dots, p_N$ are arbitrarily chosen distinct vertices on the graph.

We are now ready to delineate the major result of this paper.

\begin{Theorem}\label{t1}
	There exists a critical value $$\hat{\lambda} \ge \frac{6^6}{5^5}\frac{4\pi N}{\text{Vol}(V)} $$ such that if $\lambda > \hat{\lambda}$, then \eqref{E1} has at least two solutions, if $\lambda = \hat{\lambda}$, then \eqref{E1} has a solution, and if $\lambda < \hat{\lambda}$, then \eqref{E1} admits no solution.
\end{Theorem}

The rest of the paper is arranged as below. In section 2, we present some results that we will use in the following pages. In section 3, we prove Theorem \ref{t1}.

\section{Preliminary results}

 For each edge $xy \in E$, We suppose that its weight $w_{xy}>0$ and that $w_{xy}=w_{yx}$. Set $\mu: V \to (0,+\infty)$ be a finite measure. For any function $u: V \to \mathbb{R}$, the Laplacian of $u$ is defined by 
\begin{equation}\label{1}
	\Delta u(x)=\frac{1}{\mu(x)} \sum_{y \sim x} w_{y x}(u(y)-u(x)),
\end{equation}
where $y \sim x$ means $xy \in E$. The gradient form of $u$ reads 
\begin{equation}
	\Gamma(u, v)(x)=\frac{1}{2 \mu(x)} \sum_{y \sim x} w_{x y}(u(y)-u(x))(v(y)-v(x)).
\end{equation}
We denote the length of the gradient of $u$ by
\begin{equation*}
	|\nabla u|(x)=\sqrt{\Gamma(u)(x)}=\left(\frac{1}{2 \mu(x)} \sum_{y \sim x} w_{x y}(u(y)-u(x))^{2}\right)^{1 / 2}.
\end{equation*}
Denote, for any function $
u: V \rightarrow \mathbb{R}
$, an integral of $u$ on $V$ by $\int \limits_{V} u d \mu=\sum\limits_{x \in V} \mu(x) u(x)$. Denote  $|V|$=$ \text{Vol}(V)=\sum \limits_{x \in V} \mu(x)$ the volume of $V$. For $p \ge 1$, denote $|| u ||_{p}:=(\int \limits_{V} |u|^{p} d \mu)^{\frac{1}{p}}$. Define a sobolev space and a norm on it by 
\begin{equation*}
	W^{1,2}(V)=\left\{u: V \rightarrow \mathbb{R}: \int \limits_{V} \left(|\nabla u|^{2}+u^{2}\right) d \mu<+\infty\right\},
\end{equation*}
and \begin{equation*}
\|u\|_{H^{1}(V)}=	\|u\|_{W^{1,2}(V)}=\left(\int \limits_{V}\left(|\nabla u|^{2}+u^{2}\right) d \mu\right)^{1 / 2}.
\end{equation*}

To prove our main result, we need the following Sobolev embedding, Poincaré  inequality and Maximum principle on graphs.
\begin{Lemma}\label{21}
	{\rm (\cite[Lemma 5]{ALY})} Let $G=(V,E)$ be a finite graph. The sobolev space $W^{1,2}(V)$ is precompact. Namely, if ${u_j}$ is bounded in $W^{1,2}(V)$, then there exists some $u \in W^{1,2}(V)$ such that up to a subsequence, $u_j \to u$ in $W^{1,2}(V)$.
\end{Lemma}

\begin{Lemma}\label{2.2}
{\rm (\cite[Lemma 6]{ALY})}	Let $G = (V, E)$ be a finite graph. For all functions $u : V \to \mathbb{R}$ with $\int \limits_{V} u d\mu = 0$, there 
	exists some constant $C$ depending only on $G$ such that $\int \limits_{V} u^2 d\mu \le C \int \limits_{V} |\nabla u|^2 d\mu$.
\end{Lemma}

\begin{Lemma}\label{22}
	 {\rm (\cite[Lemma 4.1]{HLY})} Let $G=(V,E)$, where $V$ is a finite set, and $K \ge 0$ is constant. Suppose a real function $u(x): V\to \mathbb{R}$ satisfies $(\Delta-K)u(x) \ge 0$ for all $x\in \mathbb{R}$, then $u(x) \le 0$ for all $x \in V $.
\end{Lemma}

\begin{Lemma}\label{24}
		Let $G = (V, E)$ be a finite graph. For all functions $u : V \to \mathbb{R}$ with $ \int \limits_{V} u d\mu = 0$ and $p\ge 1$, there 
	exists a constant $C=C(G,p)$ such that 
	$$||u||_{p} \le C  ||\nabla u||_2 .$$
\end{Lemma}
\begin{proof}
	By Lemma \ref{2.2}, there exists $C=C(G)$ such that $$ M:=\max \limits_{x\in V} |u(x)| \le (\frac{C}{ \min \limits _{V} \mu} \int \limits_{V} |\nabla u|^2 d\mu)^{\frac{1}{2}}.$$ Thus we have $||u||_{p} \le (\frac{C}{\min \limits_{V} \mu} \int \limits_{V} |\nabla u|^2 d\mu)^{\frac{1}{2}} |V|^{1/p}:=C_2 ||\nabla u||_2$.
\end{proof}

\section{The proof of Theorem 1.1}
Since $\int \limits_{V}-\frac{4 \pi N}{\text{Vol}(V)}+4 \pi \sum \limits_{j=1}^{N} \delta p_{j} d \mu=0$, we can choose a solution $u_0$ of the equation  
\begin{equation}\label{5}
	\Delta u_{0}=-\frac{4 \pi N}{\text{Vol(V)}}+4 \pi \sum_{j=1}^{N} \delta_{p_{j}}.
\end{equation}

Letting $u=u_0+v$, the equation \eqref{E1} can be reduced to the following question
\begin{equation}\label{6}
	\Delta v=\lambda e^{u_{0}+v}\left(e^{u_{0}+v}-1\right)^5+\frac{4 \pi N}{\text{Vol}(V)}.
\end{equation}
Set $F(y):=(e^y-1)^5e^y$ on $\mathbb{R}$, it is clear that $F$ has a unique minimal value $-\frac{5^5}{6^6}.$
Thus, it follows from \eqref{6} that
$$0=\int \limits_{V} \Delta v d \mu \ge \lambda \int \limits_{V} -\frac{5^5}{6^6} d\mu +4\pi N=-\frac{5^5}{6^6}\lambda \text{Vol}(V)+ 4\pi N.$$ This implies that $$\lambda \ge \frac{6^6}{5^5} \frac{4\pi N} {|V|},$$ which is a necessary condition for the existence of solutions to \eqref{E1}. 
 
 To solve \eqref{6}, for a constant $K\ge \lambda$, we define a sequence $\{w_n\}$ by a monotone iterative scheme: 
 \begin{equation}\label{9}
 		\begin{aligned}
 			(\Delta-K) W_{n} &=\lambda e^{u_{0}+W_{n-1}}\left(e^{u_{0}+W_{n-1}}-1\right)^{5}-K W_{n-1}+\frac{4 \pi N}{\operatorname{Vol}(V)},~ n=1,2, \ldots \\
 			W_{0} &=-u_{0} .
 		\end{aligned}
 	\end{equation}
 
 Next, we establish a solution to \eqref{6} by a supersolution and subsolution method.
 \begin{Definition}\label{s}
 	A function $u$ on $V$ is called a subsolution of \eqref{6} if 
 	\begin{equation*}
 		\Delta u\ge \lambda e^{u_{0}+u}\left(e^{u_{0}+u}-1\right)^5+\frac{4 \pi N}{\operatorname{Vol}(V)}.
 	\end{equation*}
 \end{Definition}
 
 \begin{Lemma}\label{2}
 	Let ${W_n}$ be a 
 	sequence defined by scheme \eqref{9} with $K\ge\lambda$. Then 
 	\begin{equation}\label{10}
 		W_- \le \dots \le W_n \le \dots \le W_2 \le W_1 \le W_0
 	\end{equation}
 for any subsolution $W_{-}$ of \eqref{6}.
 \end{Lemma}

 \begin{proof}
 	By \eqref{5} and \eqref{9}, we have 
 	\begin{equation}
 		(\Delta-K)(W_1-W_0)=4\pi\sum_{s=1}^{N} \delta_{p_{s}}>0, x\in V.
 	\end{equation}
 By Lemma \ref{22}, we deduce that $$(W_1-W_0)(x)\le 0$$ for all $x\in V$.
 Suppose that $$W_0\ge W_1 \ge \dots \ge W_k.$$ By \eqref{9}, we conclude that 
 \begin{equation*}
 	\begin{aligned}
 		(\Delta-K)\left(W_{k+1}-W_{k}\right) &= \lambda\mathrm{e}^{u_{0}+W_{k}}\left(\mathrm{e}^{u_{0}+W_{k}}-1\right)^{5}-\lambda\mathrm{e}^{u_{0}+W_{k-1}}\left(\mathrm{e}^{u_{0}+W_{k-1}}-1\right)^{5}-K\left(W_{k}-W_{k-1}\right) \\
 		&=\left[\lambda \mathrm{e}^{u_{0}+\xi}\left(\mathrm{e}^{u_{0}+\xi}-1\right)^{4}\left(6 \mathrm{e}^{u_{0}+\xi}-1\right)-K\right]\left(W_{k}-W_{k-1}\right) \\
 		& \geq( \lambda-K)\left(W_{k}-W_{k-1}\right) \\
 		& \geq 0,
 	\end{aligned}
 \end{equation*}
 where $W_k \le \xi \le W_{k-1}$. By Lemma \ref{22}, we see that $W_{k+1} \le W_{k}$ on $V$. 
 
 Set $(W_{-}-W_0) (x_0)=\max \limits_{x\in V} (W_{-}-W_0)(x)$, we claim that $$(W_{-}-W_0)(x_0)\le 0.$$ Otherwise, $(W_{-}-W_0)(x_0)>0$. It follows that 
 $$\Delta (W_{-}-W_0)(x_0)\ge \lambda e^{u_0+W_{-}(x_0)}(e^{u_0+W_{-}(x_0)}-1))^5 +\frac{4\pi N}{\text{Vol}(V)} + \Delta u_0 >0.$$  By the definition of Laplace operator, we obtain  $$\Delta(W_{-}-W_0)(x_0)\le 0.$$ This is a contradicion. Thus we have $$(W_{-}-W_0)(x_0)\le 0,$$ which implies that $$W_{-}-W_0 \le 0$$ on $V$. Assume that $W_{-}-W_k \le 0$ for some integer $k\ge 0$. Thanks to $W_{-}$ is a subsolution of \eqref{6} and $K \ge \lambda$, we deduce that
 \begin{equation*}
 	\begin{aligned}
 		(\Delta-K)\left(W_{-}-W_{k+1}\right) & \geq \lambda\left[\mathrm{e}^{u_{0}+W_{-}}\left(\mathrm{e}^{u_{0}+W_{-}}-1\right)^{5}-\mathrm{e}^{u_{0}+W_{k}}\left(\mathrm{e}^{u_{0}+W_{k}}-1\right)^{5}\right]-K\left(W_{-}-W_{k}\right) \\
 		& \geq\left[\lambda \mathrm{e}^{u_{0}+\eta}\left(\mathrm{e}^{u_{0}+\eta}-1\right)^{4}\left(6 \mathrm{e}^{u_{0}+\eta}-1\right)-K\right]\left(W_{-}-W_{k}\right) \\
 		& \geq( \lambda-K)\left(W_{-}-W_{k}\right) \\
 		& \geq 0,
 	\end{aligned}
 \end{equation*}
 where $W_{-}\le \eta \le W_{k}$. By Lemma \ref{22}, we have $W_{-}\le W_{k+1}$ on $V$. 
 
 We now complete the proof.
 \end{proof}

\begin{Lemma}
	If $\lambda>0$ is sufficiently large, then there exists a solution of \eqref{6} on $V$.
\end{Lemma}

\begin{proof}
	Assume that $u_0$ is a solution of \eqref{6}. Select a constant $Q_0$ such that $u_0<Q_0$. Let $\hat{W}_{-}\equiv -Q_0$ on $V$, then for sufficiently large $\lambda$, we have  $$0=\Delta \hat{W}_{-}> \lambda e^{u_0+\hat{W}_{-}}( e^{u_0+ \hat{W}_{-}}-1) ^5+\frac{4\pi N}{\operatorname{Vol}(V)}.$$ 
	Thus $\hat{W}_{-}$ is a subsolution of \eqref{6}. By Lemma \ref{2}, we get a sequence $\{W_{n} \}$ satisfying 
	\begin{equation*}\label{11}
		\hat{W}_- \le \dots \le W_n \le \dots \le W_2 \le W_1 \le -u_0.
	\end{equation*}	
	Thus we can define $w(x):=\lim \limits_{n\to +\infty} W_n (x)$. Letting $n\to +\infty$ in \eqref{9}, then we know that $w$ is a solution of \eqref{6}.
\end{proof}

In order to prove Lemma \ref{3.4}, we need the following proposition.
\begin{Lemma}\label{3}
	 If $u$ is a solution of equation \eqref{E1} in $V$, then $u<0$ on $V$.
\end{Lemma}

\begin{proof}
	Suppose that $$u(x_0)=\max \limits_{x\in V} u(x),$$ we claim that $u(x_0)<0$. Suppose by way of contradiction that $u(x_0) \ge 0$, then $$e^{u(x_0)}-1 \ge 0,$$ which implies that $\Delta u(x_0) > 0$. By \eqref{1}, we have $$0\ge \Delta u(x_0).$$ This is a contradiction.
\end{proof}

\begin{Lemma}\label{3.4}
	There exists $\hat{\lambda} \ge \frac{4\pi N}{\text{Vol}(V)} \frac{6^6}{5^5}$ such that when $\lambda\ge \hat{\lambda}$, \eqref{E1} admits a solution, and when $\lambda<\hat{\lambda}$, \eqref{E1} admits no solutions.
\end{Lemma}

\begin{proof}
	Denote $A:=\{\lambda>0 | \lambda~ \text{is such that \eqref{E1} admits a solution} \} $. We claim that $A$ is a interval. If $\lambda_0 \in A$, let $v^{'}$ be the solution of \eqref{E1} with $\lambda =\lambda_0$. By Lemma \ref{3}, we have $$v{'}<0~ \text{on}~ V.$$ Set $u^{'}=v^{'}-u_0$, then 
	$$u^{'}+u_0<0~ \text{on}~V.$$ It is easy to check that $u^{'}$ is a subsolution of \eqref{6} for $\lambda \ge \lambda _0 $. It follows from Lemma \ref{2} that $\lambda \in A$ for $\lambda\ge \lambda_0$. Thus $A$ is an interval. Clearly, $$\lambda_{i}:=inf A$$ is well defined. We can choose a sequence $\{\lambda_{n} \} \subset A$ such that $\lambda_{n} \to \lambda_{i}$. On account of $ \lambda_{n}  \ge \frac{4\pi N}{\text{Vol(V)}} \frac{6^6}{5^5}$, we obtain $$\lambda_{i} \ge \frac{4\pi N}{\text{Vol}(V)} \frac{6^6}{5^5}.$$
	
	For any $\lambda> \hat{\lambda}$,  we can find a solution of \eqref{6} denoted by $u_{\lambda} (x)$. We next prove that if $\lambda_1 > \lambda_2> \hat{\lambda} $, then $u_{\lambda_1} \ge u_{\lambda_2}$ on $V$. By Lemma \ref{3}, $u_0+ u_{\lambda_2} <0$. Thus we deduce that 
	\begin{equation*}
		\begin{aligned}
			\Delta u_{ \lambda_2} &= \lambda_2 \mathrm{e}^{u_0+ u_{\lambda_2}}\left(\mathrm{e}^{u_0+ u_{\lambda_2}}-1\right)^{5}+\frac{4\pi N}{\text{Vol} (V)} \\
			&>\lambda_1 \mathrm{e}^{u_0+u_{ \lambda_2}}\left(\mathrm{e}^{u_0+ u_{\lambda_2}}-1\right)^{5}+\frac{4\pi N}{\text{Vol} (V)} 	.
		\end{aligned}
	\end{equation*}
and hence that $u_{\lambda_2}$ is a subsolution of \eqref{6} with $\lambda = \lambda_1$. By a similar argument as Lemma \ref{2}, we can show that \begin{equation}\label{101}
	u_{\lambda_2} \le u_{\lambda_1}~\text{ on } ~V.
\end{equation} 
Thus we can define $U(x):= \lim\limits_{\lambda \to \hat{\lambda}^+} u_{\lambda} (x) \in [-\infty,-u_0)$.

 We claim that
  \begin{equation}\label{3.8}
 	U(x)>-\infty ~~~\forall x\in V.
 \end{equation}  Suppose that 
			$
			\lim \limits_{\lambda \to \hat{\lambda}^+ } u_{\lambda}(x)=-\infty	
			$ for all $x\in V$.
		Integrating 	\begin{equation}\label{37} \Delta u_{\lambda}=\lambda e^{u_{0}+u_{\lambda}}\left(e^{u_{0}+u_{\lambda}}-1\right)^5+\frac{4 \pi N}{\text{Vol} (V)} 	
		\end{equation} on $V$, we obtain\begin{equation}\label{38}
				\begin{aligned}
					0=\int \limits_{V} \Delta u_{\lambda} d \mu &=\int \limits_{V} \lambda e^{u_{0}+u_{\lambda}}\left(e^{u_{0}+u_{\lambda}}-1\right)^{5} d \mu+4 \pi N \\
					&=\lambda\sum_{x\in V} \mu(x) e^{u_{0}+u_{\lambda}}\left(e^{u_{0}+u_{\lambda}}-1\right)^{5} d \mu+4 \pi N.
				\end{aligned}
			\end{equation}
			Letting $\lambda \rightarrow \hat{\lambda}^+$ in \eqref{38}, we see that $0=4\pi N$, which is a contradiction. Define 
			\begin{equation}
				V_{1}=\left\{x \in V | \lim _{\lambda \rightarrow \hat{\lambda}^+} u_{\lambda}=-\infty\right\}, V_{2}:=\left\{x \in V \mid \lim _{\lambda \rightarrow \hat{\lambda}^+} u_{\lambda} \text { exists in }(-\infty, -u_0) \right\}.
			\end{equation}
		If $V_1 = \emptyset$, then \eqref{3.8} holds. Next, we suppose that $V_1 \not = \emptyset$ and $V_2 \not= \emptyset$. Choose $y_2 \in V_2$, then 
		\begin{equation*}
			\begin{aligned}
				\Delta u_{\lambda}\left(y_{2}\right) &=\frac{1}{\mu \left(y_{2}\right)} \sum_{x \sim y_{2}} w_{x y_{2}}\left(u_{\lambda}(x)-u_{\lambda}\left(y_{2}\right)\right) \\
				&=\frac{1}{\mu \left(y_{2}\right)} \sum_{y\sim y_{2}, y_{\in} V_{1}} w_{y x_{2}}\left(u_{\lambda}(y)-u_{\lambda}\left(y_{2}\right)\right)+\frac{1}{\mu (y_{2})} \sum_{y\sim y_{2}, y_{\in} V_{2}  }w_{y x_{2}} \left(u_{\lambda}(y)-u_{\lambda}\left(y_{2}\right)\right) \\
				& =:I_1(\lambda)+I_2(\lambda).
			\end{aligned}
		\end{equation*}
	Clearly, 	$\lim \limits_{\lambda \rightarrow \hat{\lambda}^+} I_1(\lambda)=-\infty$ and $\lim \limits _{\lambda \rightarrow \hat{\lambda}^+} I_2(\lambda)$ exists in $\mathbb{R}$. By \eqref{37}, we have $$\Delta u_{\lambda} (y_2) \ge \lambda(-\frac{5^5}{6^6}) + \frac{4\pi N}{\text{Vol} (V)}.$$ This is impossible. Thus we have $V_1= \emptyset$.
	
	 Letting $\lambda \to \hat{\lambda}^+$ in 
\eqref{37},
	we can deduce that $U$ is a solution of \eqref{6} with $\lambda=\hat{\lambda}$.
\end{proof}

Define \begin{equation}\label{ene}
	I_{\lambda}(v):=\int \limits_{V} \frac{1}{2}|\nabla v|^{2}+\frac{\lambda}{6}\left(e^{u_{0}+v}-1\right)^{6}+\frac{4 \pi N}{\text{Vol} (V)} vd \mu.
\end{equation}

We may give a sufficient condition under which the problem \eqref{6} admits a solution and $I_{\lambda}(v)$ has a minimizer.
 \begin{Lemma}\label{ji}
 	If $\lambda> \hat{\lambda} $, then there exists a solution $v_\lambda$ of \eqref{6}  and it is a local minimum of the functional $I_{\lambda}(v)$ defined by \eqref{ene}.  \end{Lemma}
 \begin{proof}
 	Thanks to $u_0 + U(x) < 0$, we conclude that $U(x)$ is a subsolution of \eqref{6} for $\lambda>\hat{\lambda}$. We define $$
 	A=\left\{v \in W^{1,2} \mid v \geqslant U \text { in } V\right\} \text {. }
 	$$
 	Clearly, $I_\lambda$ is bounded from below on $V$. Thus we can define $$\eta_0 := \inf \limits_{v \in A} I_{\lambda} (v).$$
 	Set $\{v_n \}$ be a minimizing sequence and $v_n= v_n^{'}+c_n$, $n=1,2,\dots$, where 
 	$$c_n=\frac{\int \limits_{V} v_n d\mu} {\text{Vol(V)}}. $$ It is easy to see that
 	 $$c_n\ge \frac{\int \limits_{V} U d\mu} {\text{Vol(V)}}.$$ Thus, we get 
 	\begin{equation}
 		I_{\lambda}\left(v_{n}\right) \geqslant \int \limits_{V} \frac{1}{2}\left|\nabla v_{n}\right|^{2}d \mu+\frac{4 \pi N}{\text{Vol(V)}} \int \limits_{V} U d \mu,
 	\end{equation}
 which implies that $\{ \|\nabla v_n\|_2 \}_{n=1}^{\infty} $ is bounded. By \eqref{ene}, we have $$I_{\lambda}(v_n) \ge \int \limits_{V} \frac{4\pi N}{\text{Vol(V)}} c_n d\mu,$$ which implies that $$c_n \le \frac{I_{\lambda}(v_n)}{4\pi N}.$$ Thus, $\{v_n \}$ is bounded in $W^{1,2} (V)$. Since $V$ is a finite graph, by passing to a subsequence, there exists $v_{\lambda}(x)$ such that 
 
$$v_n(x) \to v_{\lambda}(x)$$ as $n\to+\infty$ for every $x\in V$. Thus $$I_{\lambda}(v_\lambda)=\eta_0.$$ By a similar argument as the appendix of \cite{T}, we can deduce that $v_\lambda $ is a solution of \eqref{6}.
 
 We next show that $v_{\lambda}> U$ in $V$. It is easy to check that 
\begin{equation}\label{m1}
	\Delta (U-v_\lambda)> \hat{\lambda} (U-v_\lambda).
\end{equation} By Lemma \ref{22}, we have $W:=U-v_\lambda \le 0$ on $V$. We claim that 
$$W(x_0):=\max \limits_{V} W <0.$$ Otherwise, $W(x_0)=0$. Clearly, $\Delta W(x_0)\le 0$. By \eqref{m1}, we obtain $W(x_0)<0$, which is a contradiction. Thus we have $$U<v_\lambda~ \text{on}~ V.$$

 We claim that $v_ \lambda$ is a local minimum of $I_{\lambda}(v)$ in $A$. For any integer $n\ge 1$, we see that \begin{equation}
 	\inf \left\{I_{\lambda}(w) \mid w \in W^{1,2}(V),\|w-v_\lambda \|_{W^{1,2}(V)} \leq \frac{1}{n}\right\}=\varepsilon_{n}<I_{\lambda}(v_\lambda).
 \end{equation}
 By a similar argument as above, we can deduce that there exists $\{v_n\} \subset W^{1,2}(V)$ satisfying $$\| v_n-v_\lambda \|_{W^{1,2}(V)} \le \frac{1}{n}$$ and $$I_{\lambda} (v_n) = \epsilon_n.$$ Thus we conclude that, by passing to a subsequence, $v_n \to v_\lambda$ in $V$ as $n\to +\infty$, and hence that $v_n>U$ for sufficiently large $n$. Therefore, we obtain $I_{\lambda}(v_n) \ge I_{\lambda} (v_\lambda)$. This is a contradiction.
 \end{proof}
 
 We now prove that $I_\lambda (v)$ satisfies the Palais-Smale condition.
 \begin{Lemma}\label{p}
 	Every sequence $\{v_n\} \subset W^{1,2} (V)$ satisfying
 	\begin{equation}\label{3.6}
 		I_{\lambda}(v_n)\to \alpha~ \text{and}~ \| I^{'}_{\lambda}(v_n) \|\to 0 ~\text{as} ~n \to +\infty
 	\end{equation} 
 has a convergent subsequence.
 \end{Lemma}
 \begin{proof}
 	From \eqref{3.6}, we have \begin{equation}\label{329}
 		\frac{1}{2}\left\|\nabla v_{n}\right\|_{2}^{2}+\frac{\lambda}{6} \int \limits_{V}\left(e^{u_{0}+v_{n}}-1\right)^{6} \mathrm{~d} x+\frac{4 \pi N}{|V|} \int \limits_{V} v_{n} \mathrm{~d} x = \alpha +o(1), \text{as}~ n\to +\infty 
 	\end{equation} 
 \begin{equation}\label{330}
 			\left|\int \limits_{V} \Gamma(v_{n},\varphi)    \mathrm{d} x+\lambda \int \limits_{V} e^{u_{0}+v_{n}}\left(e^{u_{0}+v_{n}}-1\right)^{5} \varphi \mathrm{d} x+\frac{4 \pi N}{|V|} \int \limits_{V} \varphi \mathrm{d} x\right| \leq \varepsilon_{n}\|\varphi\|_{W^{1,2}(V)},~\epsilon_n\to 0
 	\end{equation} 
 as $n\to +\infty,$ $\varphi \in H^1(V)$. By taking $\varphi=1$ in \eqref{330}, we have
 \begin{equation*}
 	\lambda \int \limits_{V} e^{u_{0}+v_{n}}\left(e^{u_{0}+v_{n}}-1\right)^{5} \mathrm{~d} x+4 \pi N \leq \varepsilon_{n}|V|^{1/2},
 \end{equation*}
from which we deduce that
\begin{equation*}
	\begin{aligned}
		\frac{\varepsilon_{n}|V|^{1/2}}{\lambda} & \geq \frac{4 \pi N}{\lambda}+\int \limits_{V} e^{u_{0}+v_{n}}\left(e^{u_{0}+v_{n}}-1\right)^{5} \mathrm{~d} \mu \\
		&=\frac{4 \pi N}{\lambda}+\int \limits_{V}\left(e^{u_{0}+v_{n}}-1\right)^{6} \mathrm{~d} \mu+\int \limits_{V}\left(e^{u_{0}+v_{n}}-1\right)^{5} \mathrm{~d} \mu \\
		& \geq \frac{4 \pi N}{\lambda}-\frac{1}{6}|V|+\frac{1}{6} \int \limits_{V}\left(e^{u_{0}+v_{n}}-1\right)^{6} \mathrm{~d} \mu .
	\end{aligned}
\end{equation*}
This implies that there exists a constant $C=C(\epsilon_n,\lambda,|V|)>0$ such that
\begin{equation}\label{318}
	\int \limits_{V}\left(e^{u_{0}+v_{n}}-1\right)^{6} \mathrm{~d} \mu \leq C.
\end{equation}
	Hence, we can find $C_2 >0$ such that \begin{equation}\label{332}
		\int \limits_{V} e^{6\left(u_{0}+v_{n}\right)} \mathrm{d} \mu=\int \limits_{V}\left[\left(e^{u_{0}+v_{n}}-1\right)+1\right]^{6} \mathrm{~d} \mu \leq 2^{6}\left[\int \limits_{V}\left(e^{u_{0}+v_{n}}-1\right)^{6} \mathrm{~d} \mu+|V|\right] \leq C_2.
	\end{equation}
Then by H$\ddot{\text{o}}$lder inequality, there exists $C_3 >0$ such that 
\begin{equation}
	\int \limits_{V} e^{2\left(u_{0}+v_{n}\right)} \mathrm{d} \mu \leq\left(\int \limits_{V} e^{6\left(u_{0}+v_{n}\right)} \mathrm{d} \mu \right)^{\frac{1}{3}}|V|^{\frac{2}{3}} \leq C_3.
\end{equation}
Similarly, $\int \limits_{V} e^{4\left(u_{0}+v_{n}\right)} \mathrm{d} \mu \leq C_4$ for a suitable constant $C_4>0$.
Decompose $v_n=v_n '+c_n$, where $\int \limits_{V} v_n ' d\mu=0$ and $c_n \in \mathbb{R}$ for $n=1,2,\dots$. Substituing it in \eqref{329}, we conclude that 
\begin{equation}
	\frac{1}{2}\left\|\nabla v_{n}^{\prime}\right\|_{2}^{2}+\frac{\lambda}{6} \int \limits_{V}\left(e^{u_{0}+v_{n}^{\prime}+c_{n}}-1\right)^{6} \mathrm{~d} \mu+4 \pi N c_{n} \rightarrow \alpha,
\end{equation}
as $n\to +\infty$, and hence that $c_n$ is bounded from above. By \eqref{329}, we see that there exists an integer $N$ such that $$\alpha-1 < I_{\lambda} (v_n)< \alpha +1$$ for $n\ge N$. This implies that \begin{equation}\label{336}
	\alpha -1<\frac{1}{2}\left\|\nabla v_{n}^{\prime}\right\|_{2}^{2}+\frac{\lambda}{6} \int \limits_{V}\left(e^{u_{0}+v_{n}^{\prime}+c_{n}}-1\right)^{6} \mathrm{~d} \mu+4 \pi N c_{n}<\alpha+1.
\end{equation}
From \eqref{318} and \eqref{336}, we conclude that \begin{equation}\label{337}
	\alpha-1+\frac{4 \lambda \pi N}{5}-\left(\frac{\lambda}{6}+\frac{\varepsilon_{n}}{5}\right)|V|<\frac{1}{2}\left\|\nabla v_{n}^{\prime}\right\|_{2}^{2}+4 \pi N c_{n}<\alpha+1.
\end{equation}

Next we show that $c_n$ is bounded from below. Taking $v_{n}^{\prime}$ in \eqref{330}, by Lemma \ref{2.2}, we can find a constant $C_5$ such that \begin{equation}
	\left\|\nabla v_{n}^{\prime}\right\|_{2}^{2}+\lambda \int \limits_{V} e^{u_{0}+v_{n}}\left(e^{u_{0}+v_{n}}-1\right)^{5} v_{n}^{\prime} \mathrm{d} \mu \leq \varepsilon_{n}\left\|v_{n}^{\prime}\right\|_{W^{1,2}(V)} \leq C_5 \varepsilon_{n}\left\|\nabla v_{n}^{\prime}\right\|_{2}.
\end{equation}
This implies that \begin{equation}\label{339}
	\begin{aligned}
		\| \nabla & v_{n}^{\prime} \|_{2}^{2}+\lambda \int \limits_{V} e^{6\left(u_{0}+c_{n}\right)}\left(e^{6 v_{n}^{\prime}}-1\right) v_{n}^{\prime} \mathrm{d} \mu \\
		\leq & \lambda \int \limits_{V} e^{6\left(u_{0}+c_{n}\right)} v_{n}^{\prime} \mathrm{d} \mu+C_5 \varepsilon_{n}\left\|\nabla v_{n}^{\prime}\right\|_{2} \\
		&+C_6 \int \limits_{V} e^{u_{0}+v_{n}}\left(e^{4\left(u_{0}+v_{n}\right)}+e^{3\left(u_{0}+v_{n}\right)}+e^{2\left(u_{0}+v_{n}\right)}+e^{u_{0}+v_{n}}+1\right)\left|v_{n}^{\prime}\right| \mathrm{d} \mu .
	\end{aligned}
\end{equation}
By Lemma
\ref{2.2}, Lemma
\ref{24} and H$\ddot{\text{o}}$lder inequality, we deduce that 
\begin{equation}
	\int \limits_{V} e^{6\left(u_{0}+c_{n}\right)} v_{n}^{\prime} \mathrm{d} \mu \leq C_7 \left\|v_{n}^{\prime}\right\|_{2} \leq C_8 \left\|\nabla v_{n}^{\prime}\right\|_{2},
\end{equation}
and
\begin{equation}
	\int \limits_{V} e^{5\left(u_{0}+v_{n}\right)}\left|v_{n}^{\prime}\right| \mathrm{d} \mu \leq\left(\int \limits_{V} e^{6\left(u_{0}+v_{n}\right)} \mathrm{d} \mu \right)^{\frac{5}{6}}\left(\int \limits_{V}\left|v_{n}^{\prime}\right|^{6} \mathrm{~d} \mu \right)^{\frac{1}{6}} \leq C_9\left\|v_{n}^{\prime}\right\|_{6} \leq C_{10} \left\|\nabla v_{n}^{\prime}\right\|_{2} 
\end{equation}
for suitable positive constants $C_7-C_{10}$.
Similarly, we can get all the other terms on the right hand side of \eqref{339} can be bounded by $\hat{C}||\nabla v_{n}^{\prime}||_2$, where $\hat{C}>0$ is a constant.
Thus, there exists constant $C_{11}>0$ such that \begin{equation}\label{340}
	\left\|\nabla v_{n}^{\prime}\right\|_{2}^{2}+\lambda \int \limits_{V} e^{6\left(u_{0}+c_{n}\right)}\left(e^{6 v_{n}^{\prime}}-1\right) v_{n}^{\prime} \mathrm{d} \mu \leq C_{11}\left\|\nabla v_{n}^{\prime}\right\|_{2}.
\end{equation}
Clearly, 
\begin{equation}
	\int \limits_{V} e^{6\left(u_{0}+c_{n}\right)}\left(e^{6 v_{n}^{\prime}}-1\right) v_{n}^{\prime} \mathrm{d} \mu \geq 0.
\end{equation}
Hence by \eqref{340}, we have $||\nabla v_{n}^{\prime}||_2 \le C_{12}$ for a suitable constant $C_{12}>0$. Therefore, by \eqref{337}, we deduce that $c_n$ is bounded from below. 

Thus $\{v_n\}$ is bounded in $H^{1}(V)$. Thus, there exists $v\in H^{1}(V)$ such that, by passing to a subsequence, $v_n(x) \to v(x)$ for all $x\in V$. 
 
Next, we find the second solution of \eqref{6}. From now on, we suppose that $v_\lambda$ is the local minimum as defined by Lemma 
\ref{ji}(if not, we could have already found our second solution). Thus there exists $\rho_0 >0 $ such that
 $$I_{\lambda} (v_{\lambda})\le I_{\lambda} (v)$$ for all $v:||v-v_\lambda||_{H^{1}(V)} \le \rho_0.$ For $c>0$, we have \begin{equation}\label{ss}
	\begin{aligned}
		I_{\lambda}\left(v_{\lambda}-c\right)-I_{\lambda}\left(v_{\lambda}\right)=& \frac{\lambda}{6} \int \limits_{V}\left[\left(e^{u_{0}+v_{\lambda}-c}-1\right)^{6}-\left(e^{u_{0}+v_{\lambda}-1}\right)^{6}\right] d\mu-4 \pi N c \\
		&< \frac{\lambda}{6}|V| C_{13} -4 \pi N c \rightarrow-\infty \text { as } c \rightarrow+\infty .
	\end{aligned}
\end{equation}
 There are two possibilities: (I) $v_\lambda$ is not a strict local minimum for $I_\lambda$, (II) $v_\lambda$ is a strict local minimum for $I_\lambda.$ If case (I) happens, then we deduce that
 $$
 \inf \limits_{\left\|v-v_{\lambda}\right\| _{H^{1}(V)} = \rho} I_{\lambda}=I_{\lambda}\left(v_{\lambda}\right)=:\alpha_{\lambda}$$ for all $0<\rho< \rho_0.$ It follows that there exists a local minimum $v_\rho \in H^{1}(V)$ such that $$||v_\rho- v_\lambda||= \rho,$$ and $I_\lambda(v_\rho)=\alpha_\lambda$ for all $\rho \in (0,\rho_0)$. Therefore, in this situation, we get a one-parameter family of solutions of \eqref{6}. If case (II) happens, we can find $\rho_1 \in(0, \rho_0)$ such that 
\begin{equation}\label{s1}
	\inf _{\left\|v-v_{\lambda}\right\|  _{H^{1}(V)}=\rho_{1}} I_{\lambda} (v)>I_{\lambda}\left(v_{\lambda}\right)=\alpha_{\lambda}.
\end{equation}
By \eqref{ss}, we deduce that 
$$
 I_{\lambda}\left(u_{\lambda}-c_{0}\right) \leqslant I_{\lambda}\left(u_{\lambda}\right)-1<I_{\lambda}\left(v_{\lambda}\right)
 $$
 for some $c_0>|V|^{-\frac{1}{2}} \rho_{1}$.
 
 Define
  $\mathcal{P}=\left\{\gamma:[0,1] \rightarrow H^{1}(V)\right.$ $| \gamma$ is continuous and satisfies $\left.\gamma(0)=v_{\lambda}, \gamma(1)=v_{\lambda}-c_{0}\right\}$ and
  $$\alpha=\inf \limits _{\gamma \in \mathcal{P}} \sup \limits _{t\in [0,1]} I_{\lambda}(\gamma(t)) .$$
 From \eqref{s1},
 we conclude that $$\alpha>I_{\lambda}\left(v_{\lambda}\right) \geqslant \max \left\{I_{\lambda}(\gamma(0)), I_{\lambda}(\gamma(1))\right\}~ \forall \gamma \in \mathscr{P}.$$ Thus, by Lemma \ref{p}
 , $I_\lambda$ satisfies the hypothesis of the mountain-pass theorem. Thus $\alpha$ is a critical point of $I_\lambda$. By virtue of $\alpha > I_{\lambda} (v_\lambda)$, we get a second solution of \eqref{6}. 
 
 We now complete the proof of Theorem \ref{t1}.
  \end{proof}
 
\begin{proof}[Acknowledgements.]	
		The authors would like to thank the anonymous Referees for their valuable comments
	which helped to improve the manuscript.
\end{proof}

\end{document}